\documentclass[12pt,reqno]{amsart}

\usepackage{amsthm,amssymb,amstext,amscd,euscript,mathrsfs,amsfonts,amsbsy,amsxtra,latexsym,amsmath}
\usepackage{fullpage}
\usepackage[english]{babel}
\usepackage[latin1]{inputenc}
\usepackage{verbatim}
\usepackage{graphicx} 
\usepackage[all]{xy} 
\usepackage{enumitem}
\usepackage{bbm}
\usepackage{marvosym}
\numberwithin{figure}{section} 
\allowdisplaybreaks
\usepackage{chngcntr}
\counterwithin{table}{section}
\usepackage{booktabs}
\usepackage{colortbl}
\usepackage{xcolor}
\usepackage{graphicx}
\newcommand{\field}[1]{\mathbb{#1}} 
\newcommand{\N}{\field{N}}
\newcommand{\Z}{\field{Z}} 

\newcommand{\C}{\field{C}}

\numberwithin{equation}{section}
\newtheorem{theorem}{\textbf{Theorem}}
\numberwithin{theorem}{section}
\newtheorem{corollary}[theorem]{\textbf{Corollary}}
\newtheorem{lemma}[theorem]{\textbf{Lemma}}

\theoremstyle{definition}

\newtheorem{remark}{Remark}[section]
\newtheorem{example}[theorem]{Example}

 \definecolor{Ftitle}{RGB}{11,46,108}
\definecolor{line}{RGB}{87,39,117}
\colorlet{tableheadcolor}{Ftitle!25} 
\colorlet{tablerowcolor}{gray!10} 
\newcommand{\bea}{\begin{eqnarray}} 
\newcommand{\eea}{\end{eqnarray}} 
\newcommand{\be}{\begin{equation}} 
\newcommand{\ee}{\end{equation}} 
\newcommand{\benn}{\begin{equation*}} 
\newcommand{\eenn}{\end{equation*}}

\usepackage[ pdfborderstyle={/S/U/W 1}]{hyperref}
\definecolor{MyBlue}{HTML}{210cac}
\definecolor{MyCiteColor}{HTML}{0099FF}
\definecolor{MyRed}{HTML}{CC033C}
\hypersetup{%
  pdfborderstyle={/S/U/W 1},
  colorlinks=true,
  linkcolor=MyBlue,
  citecolor=MyCiteColor,
  urlcolor=MyRed,
  anchorcolor=MyBlue,
  linkbordercolor=MyRed,
}

\title{On the Euler characteristic of a relative hypersurface}
\author{James Fullwood}
\address{School of Mathematical Sciences\\Shanghai Jiao Tong University\\ 800 Dongchuan Road, Shanghai, China}
\email{fullwood@sjtu.edu.cn}
\author{Martin Helmer}
\address{Department of Mathematical Sciences\\University of Copenhagen\\Universitetsparken 5, DK-2100 Copenhagen, Denmark}
\email{m.helmer@math.ku.dk}
\begin{document}

\maketitle

\begin{abstract}
We derive a general formula for the Euler characteristic of a fibration of projective hypersurfaces in terms of invariants of an arbitrary base variety. When the general fiber is an elliptic curve, such formulas have appeared in the physics literature in the context of calculating D-brane charge for M-/F-theory and type-IIB compactifications of string vacua.  While there are various methods for computing Euler characteristics of algebraic varieties, we prove a base-independent pushforward formula which reduces the computation of the Euler characteristic of relative hypersurfaces to simple algebraic manipulations of rational expressions determined by its divisor class in a projective bundle. We illustrate our methods by applying them to an explicit family of relative hypersurfaces whose fibers are of arbitrary dimension and degree.     

\end{abstract}

\section{Introduction}\label{intro}

Viewing algebraic geometry from a relative perspective has played a crucial role in its modern development. From such a viewpoint the fundamental objects are morphisms between varieties, and the static objects of classical algebraic geometry are associated with morphisms to a point. In particular, given an equation defining a classical projective variety, such as
\begin{equation}\label{e1}
y^2z=x^3+fxz^2+gz^3 \quad f,g\in \mathbb{C},
\end{equation}
the relative viewpoint sees the numbers $f$ and $g$ as sections of line bundles over a point, so that we may promote the status of equation (\ref{e1}) to relative algebraic geometry by viewing $f$ and $g$ as sections of line bundles over an arbitrary smooth base variety. As such, we may view equation (\ref{e1}) as an equation for a hypersurface not necessarily in $\mathbb{P}^2$, but in a $\mathbb{P}^2$-\emph{bundle} over a smooth base variety $X$, giving rise to an elliptic fibration $Y_X\to X$, where the fiber over a point $p\in X$ is given by
\[
y^2z=x^3+f(p)xz^2+g(p)z^3.
\] 
As this construction makes no use of any particular features of the base variety $X$, we may view equation $(\ref{e1})$ as a map which takes a smooth variety $X$ to the elliptic fibration $Y_X\to X$. Just as the the fibration $Y_X$ is determined by $X$, certain invariants of $Y_X$ are determined by a combination of invariants of $X$ and invariants of the line bundles for which $f$ and $g$ are sections of. For example, the Euler characteristic of $Y_X$ -- which is the fundamental dimensionless quantity associated with any space -- will always be a certain $\Z$-linear combination of products of Chern classes of $X$ with the first Chern class of a line bundle $\mathscr{L}\to X$. 

The purpose of this note is to compute formulas for the Euler characteristic for relative hypersurfaces such as $Y_X$ solely in terms of invariants of an arbitrary smooth base variety $X$. Such formulas are sometimes called `Sethi-Vafa-Witten' formulas \cite{AE1}\cite{AE2}\cite{EFY}\cite{F}, since if $X$ is a Fano threefold one may construct $Y_X\to X$ such that the total space $Y_X$ is a Calabi-Yau elliptic fourfold for which Sethi, Vafa and Witten -- who were motivated by constructing super-symmetric compactifications of M- and F-theory \cite[Formula (2.12)]{sethi1996constraints} -- derived a formula for the Euler characteristic of $Y_X$ in terms of the Chern classes of $X$, namely
\begin{equation}\label{e2}
\chi(Y_X)=\int_X 12c_1(X)c_2(X)+360c_1(X)^3.
\end{equation}
While the RHS of formula (\ref{e2}) depends on the fact that in this particular case $X$ is of dimension 3, Aluffi and Esole showed one can compute a formula for the Euler characteristic for the relative variety $Y_X$ for all possible $X$ at once \cite[Corollary~4.3]{AE1}, yielding the formula 
\begin{equation}\label{e3}
\chi(Y_X)=\int_X 12c_1(X)\sum_{i=0}^{\text{dim}(X)-1} c_i(X)\left(-6c_1(X)\right)^{\text{dim}(X)-1-i}.
\end{equation}
What we concern ourselves with in this paper is then deriving a genralization of \eqref{e3} for when the fiber is not necessarily an elliptic curve, but a genreral projective hypersurface of arbitrary dimension and degree.

Along with providing a general procedure for obtaining such formulas for the Euler characteristic of relative varieties such as $Y_X$, we show that such Euler characteristic formulas actually contain the information of the \emph{total} Chern class of the relative variety pushed forward to the base. In regards to the elliptic fibration $Y_X$, if we denote the projection of  $Y_X$ to $X$ by $\varphi:Y_X\to X$, we show (as a consequence of Theorem~\ref{rcc})
\begin{equation}\label{e4}
\varphi_*c(Y_X)=\sum_{j=1}^{\text{dim}(X)}\left(12c_1(X)\sum_{i=0}^{j-1} c_i(X)\left(-6c_1(X)\right)^{j-1-i}\right),
\end{equation}
where $\varphi_*c(Y_X)$ denotes the proper pushforward to $X$ of the total Chern class of $Y_X$. The index $j$ in formula (\ref{e4}) then corresponds to the the component of $\varphi_*c(Y_X)$ of codimension $j$ in $X$, so that all components of $\varphi_*c(Y_X)$ are essentially truncations of the Euler characteristic formula (\ref{e3}). Moreover, in keeping with the relative perspective, there is a functorial theory of Chern classes for which $\varphi_*c(Y_X)$ is precisely the Chern class of the \emph{morphism} $\varphi:Y_X\to X$.

Our approach to Chern classes is rooted in Fulton-MacPherson intersection theory, so that the Chern class of a variety is taken to be an element of its Chow group of algebraic cycles modulo rational equivalence. Since from a computational perspective we are primarily concerned with computing pushforwards of Chern classes, the essential tool for our computations is a pushforward formula for rational equivalence classes in a projective bundle which we prove in \S\ref{pf}. Our formula is essentially a generalization to projective bundles of arbitrary rank of a pushforward formula appearing in \cite[Theorem~2.7]{EJY} for classes in a $\mathbb{P}^2$-bundle, and a special case of our pushforward formula also appears in \cite[Theorem~2.5]{F} (for which this paper is a natural continuation of). 

The paper is organized as follows. In \S\ref{rv} we illustrate a general procedure for embedding relative projective varieties in a projective bundle by constructing a relative hypersurface $Z_X(n,d)\to X$ of degree $d$ in a $\mathbb{P}^n$-bundle (for arbitrary integers $d\geq 2$ and $n\geq 1$), for which $Z_X(2,3)$ may serve as an $F$-theory compactification of string vacua which has yet to appear in the literature. In \S\ref{fcc} we review the theory of Chern classes from a functorial perspective and how it lends itself to a natural extension of the Gauss-Bonnet-Chern theorem to the relative setting. The pushforward formula is then proved in \S\ref{pf}. In \S\ref{rccc} we define the Chern class of a morphism, and then use the functorial formalism introduced in \S\ref{fcc} to compute the Chern class of the morphism $Z_X(n,d)\to X$ constructed in \S\ref{rv}. While the aforementioned Chern class computation in \S\ref{rccc} relies on the contribution of singular fibers of $Z_X(n,d)\to X$ and singular Chern classes, in \S\ref{SVW} we apply the pushforward formula proved in \S\ref{pf} to bypass the singularity-theoretic methods used in \S\ref{rccc} and compute the Chern class of the morphism associated with an arbitrary relative hypersurface in a $\mathbb{P}^n$-bundle via simple algebraic manipulations, yielding Sethi-Vafa-Witten formulas analogous to (\ref{e3}).  

For the rest of the paper we work over the complex numbers $\mathbb{C}$, though pretty much everything we do works over an arbitrary algebraically closed field of characteristic zero.

\section{Constructing a relative variety}\label{rv}
As varieties are typically embedded in either affine or projective space, relative varieties are typically embedded in either relative affine or projective spaces, which are vector bundles and projective bundles (respectively) over a smooth base variety. For the sake of compactness we will restrict ourselves to the projective case, so that our relative varieties will be embedded in projective bundles over a smooth compact base variety. One may construct projective bundles either as the space of fiber-wise lines or fiber-wise hyperplanes passing through the zero-section of a vector bundle. In what follows we will choose the former alternative.

Let $X$ be a smooth compact variety over $\mathbb{C}$. For every complex vector bundle $\mathscr{E}\to X$ of rank $n+1$ we let $\mathbb{P}(\mathscr{E})\to X$ denote the projective bundle of \emph{lines} in $\mathscr{E}$, which is a $\mathbb{P}^n$-bundle over $X$. The projective bundle comes with a tautological line bundle corresponding to the fiber-wise lines in $\mathscr{E}$ which we denote by $\mathscr{O}(-1)$. We then denote its dual by $\mathscr{O}(1)$, whose restriction to each $\mathbb{P}^n$ fiber of $\mathbb{P}(\mathscr{E})\to X$ is the line bundle associated with a hyperplane. Given a homogeneous equation in $n+1$ variables which defines a hypersurface in $\mathbb{P}^n$, we may promote its numerical coefficients to sections of line bundles over $X$ in such a way that the homogeneous equation defines a hypersurface $Y_X\hookrightarrow \mathbb{P}(\mathscr{E})$ in a projective bundle $\mathbb{P}(\mathscr{E})\to X$. In such a case we refer to $Y_X$ as a \emph{projecto-relative} hypersurface.      

For the remainder of this section let $d\geq 2$ and $n\geq 1$ be arbitrary integers. We now construct an example of a projecto-relative hypersurface of arbitrary degree in a $\mathbb{P}^n$-bundle $\mathbb{P}(\mathscr{E})\to X$ (the vector bundle $\mathscr{E}$ will be determined by our construction). For this, we use the equation 
\begin{equation}\label{e5}
x_1^d+\cdots+x_n^d+fx_1x_0^{d-1}+gx_0^d=0.
\end{equation}
When $f$ and $g$ are complex numbers, equation (\ref{e5}) defines a smooth hypersurface of degree $d$ in $\mathbb{P}^n$ whenever 
\[
(d-1)^{d-1}f^d+d(-d)^{d-1}g^{d-1}\neq 0,
\]
and a singular hypersurface whenever 
\[
(d-1)^{d-1}f^d+d(-d)^{d-1}g^{d-1}=0. 
\]
A projecto-relative hypersurface associated with equation (\ref{e5}) then enables us to bring together smooth and singular manifestations of equation (\ref{e5}) into a single family. In particular, the homogeneous coordinates $x_0,x_1,\ldots,x_n$ of $\mathbb{P}^n$ will now be homogeneous coordinates in the fibers of a $\mathbb{P}^n$-bundle $\pi:\mathbb{P}(\mathscr{E})\to X$. Each $x_i$ will then be a section of $\mathscr{O}(1)\otimes \pi^*\mathscr{L}_i$, where $\mathscr{L}_i\to X$ is an ample line bundle on $X$ which keeps track of how the coordinate $x_i$ twists from fiber to fiber in $\pi:\mathbb{P}(\mathscr{E})\to X$. And for equation (\ref{e5}) to give a well defined zero-locus of a section of a line bundle on $\mathbb{P}(\mathscr{E})$, each monomial on the LHS of equation (\ref{e5}) must be a section of the same line bundle. Since for example $x_i^d$ is a section of $\mathscr{O}(d)\otimes \pi^*\mathscr{L}_i^d$, we take $\mathscr{L}_i=\mathscr{L}$ for some fixed ample line bundle $\mathscr{L}\to X$ for $i=1,...,n$. Now after taking $x_0$ to be a section of $\mathscr{O}(1)$, and $f$ and $g$ to be sections of $\pi^*\mathscr{L}^{d-1}$ and $\pi^*\mathscr{L}^{d}$ respectively, every monomial of (\ref{e5}) is then a well-defined section of $\mathscr{O}(d)\otimes \pi^*\mathscr{L}^d$. The aforementioned prescriptions are then consistent with $\mathscr{E}\to X$ being given by
\[
\mathscr{E}=\mathscr{O}\oplus \bigoplus_{i=1}^n \mathscr{L}
\]
(so that each $x_i$ corresponds to a summand of $\mathscr{E}$). Equation (\ref{e5}) then defines a projecto-relative hypersurface $Z_X(n,d)\hookrightarrow \mathbb{P}(\mathscr{E})$, and composing the embedding with the bundle projection $\pi:\mathbb{P}(\mathscr{E})\to X$ yields an associated projection morphism $\varphi:Z_X(n,d)\to X$, whose fibers are degree $d$ hypersurfaces in $\mathbb{P}^n$. The singular fibers of $\varphi:Z_X(n,d)\to X$ lie over the (singuar) hypersurface
\[
\Delta_d:((d-1)^{d-1}f^d+d(-d)^{d-1}g^{d-1}=0)\subset X.
\]
To ensure that the total space of the relative variety $Z_X(n,d)$ is smooth we must assume that the hypersurfaces $\{f=0\}$ and $\{g=0\}$ are smooth and intersect transversally. We note that a general fiber of $Z_X(n,d)$ is Fano for $d<n+1$, Calabi-Yau for $d=n+1$, and of general type for $d>n+1$. 

To construct relative varieties defined by more than one equation the procedure is more or less the same. Often times however, one has to use more than a single line bundle to tensor the projective coordinates with. An explicit example of a relative complete intersection was constructed in \cite{EFY}. 

\begin{remark} The fibration $Z_X(2,3)\to X$ is an elliptic fibration given by
\[
Z_X(2,3):(x^3+y^3+fxz^2+gz^3=0) \subset \mathbb{P}(\mathscr{O}\oplus \mathscr{L}\oplus \mathscr{L}),
\]
whose discriminant is given by
\[
\Delta:(4f^3+27g^2=0)\subset X.
\]
If $X$ is Fano and $\mathscr{L}=\mathscr{O}(-K_X)$, $Z_X(2,3)$ has vanishing first Chern class, thus in such a case $Z_X(2,3)$ is in fact Calabi-Yau. As such, $Z_X(2,3)$ then serves as a natural compactification of $F$-theory which may be viewed as a specialization of the $E_6$ fibration defined in \cite{AE2}. However, $Z_X(2,3)$ has a much simpler singular fiber structure than $E_6$ fibrations, as its general singular fiber over $\Delta$ is a cuspidal cubic, which enhances over $\{f=g=0\}\subset \Delta$ to a bouquet of 3 $\mathbb{P}^1$s meeting at a point. Moreover, all smooth fibers of $Z_X(2,3)$ have $j$-invariant 0, which makes $Z_X(2,3)$ a very special case of the $E_6$ family. 
\end{remark}

\section{Functorial Chern classes}\label{fcc}
In the 1970s MacPherson constructed a natural transformation of covariant functors which yields a functorial theory of Chern classes for (possibly singular) varieties \cite{RMCC}. In particular, let $F$ denote the constructible function functor which takes a variety $X$ to its group of constructible functions 
\[
F(X)=\left\{\sum_i \mathbbm{1}_{W_i} \right\}
\] 
(where $\mathbbm{1}_{W_i}$ is the indicator function of a closed subvariety $W_i\subset X$), and let $A_*$ denote the functor which takes a variety to its Chow group of algebraic cycles modulo rational equivalence $A_*X$. MacPherson then constructed a natural transformation\footnote{MacPherson actually used the integral homology functor $H_*$, but everything carries over to Chow homology in a straightforward manner.} 
\[
c_*:F\to A_*,
\]
such that for smooth $X$ we have
\begin{equation}\label{nc}
c_*(\mathbbm{1}_X)=c(TX)\cap [X]\in A_*X.
\end{equation}
For singular $X$, the definition of $c_*(\mathbbm{1}_X)$ is given in terms of a $\Z$-linear combination of Chern-Mather classes of subvarieties of $X$ determined by its singularities, but the precise definition will not be needed for our purposes. 

From here on, given a constructible function $\delta\in F(X)$ we denote $c_*(\delta)$ by $c_{\text{SM}}(\delta)$, and if $\delta=\mathbbm{1}_W$ for some closed subvariety $W$ of $X$ we denote it by $c_{\text{SM}}(W)$, and refer to it as the \emph{Chern-Scwartz-MacPherson class} of $W$ (or simply `CSM class' for short). If $W\subset X$ is smooth, in light of equation (\ref{nc}) we will use the usual notation for Chern classes and denote its CSM class simply by $c(W)$.

Given a proper morphism $f:X\to Y$ and a constructible function $\mathbbm{1}_W$ for some closed subvariety $W\subset X$, its pushforward $f_*\mathbbm{1}_W$ is the  constructible function on $Y$ given by
\begin{equation}\label{pffcf}
f_*\mathbbm{1}_W(p)=\chi(f^{-1}(p)\cap W),
\end{equation}
where $\chi:\text{Var}_{\C}\to \mathbb{Z}$ denotes the map which takes a complex variety to its topological Euler characteristic with compact support. This definition then extends to arbitrary constructible functions by linearity. If $\text{dim}(W)=\text{dim}(f(W))$, the field of rational functions on $f(W)$ is a finite extension of the field of rational functions on $W$, whose degree we denote by $k$. The proper pushforward of the rational equivalence class $[W]\in A_*X$ is then given by
\[
f_*[W]=
\begin{cases}
k[f(W)] \quad \text{if } \text{dim}(W)=\text{dim}(f(W)) \\
\quad 0 \quad \hspace{2.2cm} \text{otherwise}. \\
\end{cases}
\]
This definition also extends to arbitrary rational equivalence classes of algebraic cycles by linearity. We emphasize that the assumption $f$ is proper -- so that $f$ takes closed subvarieties to closed subvarieties -- was essential to both notions of pushforward associated with the covariant functors $F$ and $A_*$. With these prescriptions MacPherson's natural transformation then generalizes the Gauss-Bonnet-Chern Theorem to arbitrary varieties. To see this, first note that if $f:X\to \star$ denotes the map from $X$ to a point, we have
\begin{eqnarray*}
f_*c_{\text{SM}}(X)&=&f_*c_*(\mathbbm{1}_X) \\
                              &=&c_*f_*\mathbbm{1}_X \\
                              &=&c_*\left(\chi(f^{-1}(\star)\cap X)\mathbbm{1}_{\star}\right) \\
                              &=&\chi(X)c_*(\mathbbm{1}_{\star}) \\
                              &=&\chi(X)[\star],
\end{eqnarray*}
so that $c_{\text{SM}}(X)$ encodes the Euler characteristic of $X$ in dimension zero. As such, we denote the pushforward associated with the map $f$ by $\int_X:A_*X\to A_*\star=\mathbb{Z}$, so that we may rewrite the result of the preceding computation as
\begin{equation}\label{gb}
\int_X c_{\text{SM}}(X)=\chi(X).
\end{equation}
It follows that if $Y_X$ is a relative variety with associated morphism $\varphi:Y_X\to X$, the pushforward to a point $Y_X\to \star$ factors through $\varphi$, thus by functoriality and equation (\ref{gb}) we have
\[
\int_{Y_X} c_{\text{SM}}(Y_X)=\int_X\varphi_*c_{\text{SM}}(Y_X)=\chi(Y_X),
\]
a fact which will be key moving forward. In particular, we use the `relative' Gauss-Bonnet-Chern formula
\begin{equation}\label{svwf}
\int_X\varphi_*c_{\text{SM}}(Y_X)=\chi(Y_X)
\end{equation}
to derive a formula for the Euler characteristic of $Y_X$ in terms of Chow classes on $X$, thus we need to effectively compute $\varphi_*c(Y_X)$ (since we will take $Y_X$ to be smooth). A projecto-relative hypersurface $Y_X$ is embedded in a projective bundle, thus we have the following diagram
\[
\xymatrix{
Y_X \ar[dr]_{\varphi}  \ar[r]^{\iota}  & \mathbb{P}(\mathscr{E}) \ar[d]^\pi\\
 & X ,\\
}
\]
where $\iota:Y_X\to \mathbb{P}(\mathscr{E})$ is the inclusion and $\pi:\mathbb{P}(\mathscr{E})\to X$ is the bundle projection. As such, we have 
\[
\varphi_*c(Y_X)=\pi_*(\iota_*c(Y_X)),
\]
so what we need to effectively compute is $\pi_*$, which is the focus of the next section.
 
\section{A pushforward formula}\label{pf}
Let $X$ be a smooth compact variety over $\C$ endowed with a vector bundle $\mathscr{E}\to X$ of rank $r>0$, and let $\pi:\mathbb{P}(\mathscr{E})\to X$ denote the projective bundle of \emph{lines} in $\mathscr{E}$. In this section we prove a formula for the proper pushforward map $\pi_*:A_*\mathbb{P}(\mathscr{E})\to A_*X$ associated with with the bundle projection $\pi:\mathbb{P}(\mathscr{E})\to X$. Such a map will enable us to effectively compute $\pi_*$ in a way that is independent of the dimension of $X$, and is relative in the sense that it only depends on the rank of $\mathscr{E}\to X$. We denote the dual of tautological line bundle $\mathscr{O}(-1)\to \mathbb{P}(\mathscr{E})$ by $\mathscr{O}(1)\to \mathbb{P}(\mathscr{E})$, and its associated Chow class in $A_*\mathbb{P}(\mathscr{E})$ we denote by $H$. Crucial to our computations will be the well-known formula of Grothendieck for the Chow group of a projective bundle, namely  
\begin{equation}\label{pbf}
A_*\mathbb{P}(\mathscr{E})\cong A_*X[H]/\left(H^{r}+c_1(\mathscr{E})H^{r-1}+\cdots+c_r(\mathscr{E})\right),
\end{equation}
so that any class in $A_*\mathbb{P}(\mathscr{E})$ may be written as a polynomial in $H$ with coefficients in $A_*X$ (we note that since $X$ is smooth the Chow groups $A_*\mathbb{P}(\mathscr{E})$ and $A_*X$ are in fact rings). As such, if we let $H^0=[\mathbb{P}(\mathscr{E})]$, and $\alpha \in A_*\mathbb{P}(\mathscr{E})$ is expanded as
\begin{equation}\label{e91}
\alpha=\alpha_0H^0+\alpha_1 H+\alpha_2 H^2+\cdots ,
\end{equation}
then by the projection formula we have
\[
\pi_*\alpha=\alpha_0\pi_*H^0+\alpha_1\pi_*H+\alpha_2\pi_*H^2+\cdots ,
\]
thus $\pi_*$ is determined by how it acts on powers of $H$. As we work over a base $X$ of arbitrary dimension, what our formula then accomplishes is pushing forward \emph{all} powers of $H$ at once, thus allowing one to compute $\pi_*$ in a relative setting.

We recall that the \emph{Chern roots} of $\mathscr{E}$, denoted $L_1,...,L_r$, are the formal roots of the Chern class of $\mathscr{E}$, so that
\[
c(\mathscr{E})=(1+L_1)\cdots (1+L_r)\in A_*X.
\] 
We refer to the non-zero Chern roots of $\mathscr{E}$ as the \emph{non-trivial Chern roots} of $\mathscr{E}$, and since $\mathbb{P}(\mathscr{E})$ is isomorphic to $\mathbb{P}(\mathscr{E}\otimes \mathscr{L})$ for any line bundle $\mathscr{L}\to X$, we may assume without any loss of generality that at least one of the Chern roots of $\mathscr{E}$ is zero. We remark that if $\mathscr{E}$ is in fact the trivial bundle, i.e., if all the Chern roots of $\mathscr{E}$ are zero, then by \cite{IT} \S3.2 it follows 
\[
\pi_*H^i=\begin{cases}
1 \quad \text{if } i=r-1 \\
0 \quad \text{otherwise},
\end{cases}
\]
thus we assume from here on that $\mathscr{E}$ is not the trivial bundle 
of rank $r$, as our formula will be given in terms of the non-trivial Chern roots of $\mathscr{E}$. Alternative notational conventions would allow one to incorporate the case of the trivial bundle into the theorem below (see Remark \ref{remark:OtherNotationalConvention}), however we have opted to state this case separately and to assume the bundle $\mathscr{E}$ is non-trivial.
\begin{theorem}\label{pff}
Let $L_1,\ldots ,L_m$ be the distinct non-trivial Chern roots of $\mathscr{E}$, $k_i\in \N$ be such that the multiplicity of $L_i$ is $k_i+1$, and let $\alpha\in A_*\mathbb{P}(\mathscr{E})$ be given by
\[
\alpha=\alpha_0+\alpha_1H+\alpha_2H^2+\cdots.
\]
The map $\pi_*:A_*\mathbb{P}(\mathscr{E})\to A_*X$ is then given by 
\begin{equation}\label{mf1}
\pi_*\alpha=\left.\left(D\cdot \sum_{i=1}^m g_{\alpha}(x_i)\right)\right|_{
x_1=-L_1,\dots,x_m=-L_m}\in A_*X,
\end{equation}
where 
\[
g_{\alpha}(x_i)=\frac{\alpha(x_i) -\left(\alpha_0+\alpha_1 x_i+\cdots + \alpha_{\emph{rank}(\mathscr{E})-2}x_i^{\emph{\text{rank}}(\mathscr{E})-2} \right)}{ x_i^{\emph{rank}(\mathscr{E})-m}\prod_{l=1, \; l \neq i}^{m} (x_i-x_l)},
\]
and $D$ is the linear operator on $A_*X[x_1,\ldots,x_m]$ given by
\[
g\mapsto \frac{1}{k_1!\cdots k_m!}\frac{\partial^{k_1+\cdots +k_m}}{\partial x_1^{k_1}\cdots \partial x_m^{k_m}}(x_1^{k_1}\cdots x_m^{k_m}\cdot g).
\]
\end{theorem}
\begin{proof}
By the projection formula for intersection products we have
\[
\pi_*\alpha =\sum_{j=0} \alpha_j\cdot \pi_*(H^j),
\]
thus $\pi_*\alpha$ is completely determined by the $\pi_*(H^j)$. For this, we use the fact that by Fulton's definition of Chern class (\cite{IT}, \S3.2) we have
\begin{equation}\label{scf}
\pi_*\left(1+H+H^2+\cdots\right)=\frac{1}{c(\mathscr{E})}=\frac{1}{(1+L_1)^{k_1+1}\cdots (1+L_m)^{k_m+1}}\in A_*X,
\end{equation}
where we recall that the $L_i$ are the \emph{distinct} non-trivial Chern roots of $\mathscr{E}$. As the proper pushforward of a pure dimensional cycle (not in the kernel of $\pi_*$) preserves its dimension, it follows that $\pi_*(H^j)$ coincides with the term of dimension $\text{dim}(\mathbb{P}(\mathscr{E}))-j$ in the formal series expansion of $\left((1+L_1)^{k_1+1}\cdots (1+L_m)^{k_m+1}\right)^{-1}$ (so that $\pi_*(H^j)=0$ for $j<\text{rank}(\mathscr{E})-1$). Now since
\[
\frac{1}{(1+L_i)^{k_i+1}}=\frac{1}{k_i!}\frac{d^{k_i}}{dx_i^{k_i}}\left.\left(\frac{1}{1-x_i}\right)\right|_{x_i=-L_i}= \left.\left(\sum_{j=0} \frac{(k_i+j)!}{k_i!j!}x_i^j\right)\right|_{x_i=-L_i},
\]
the term of codimension $d$ in $X$ of $c(\mathscr{E})^{-1}$ is precisely
\[
\sum_{i_1+\cdots+i_m=d} \beta_{i_1}^{(k_1)}(-L_1)^{i_1}\cdot \beta_{i_2}^{(k_2)}(-L_2)^{i_2} \cdots \beta_{i_m}^{(k_m)}(-L_m)^{i_m}, 
\]
where
\[
\beta^{(k_l)}_{i_l}=\frac{(k_l+i_l)!}{k_l!i_l!}={k_l+i_l \choose k_l}. 
\]
Now let $n=\text{rank}(\mathscr{E})-1$. It then follows that $\text{dim}(H^n)=\text{dim}(X)$, so that the term $\pi_*(H^{n+j})$ is of codimension $j$ in $X$. We then have 
\begin{equation}
\pi_*(H^{n+j})= \left\lbrace \begin{array}{l}
0 \;\;\; \hspace{9.29cm} \mathrm{for} \;\;\; j<0 \\
\displaystyle \sum_{i_1+\cdots+i_m=j} \beta_{i_1}^{(k_1)}(-L_1)^{i_1}\cdot \beta_{i_2}^{(k_2)}(-L_2)^{i_2} \cdots \beta_{i_m}^{(k_m)}(-L_m)^{i_m} \;\;\; \mathrm{for} \;\;\; j \geq 0.
\end{array} \right. \label{eq:piH12}
\end{equation}
Thus 
\begin{align*}
\pi_*\alpha=&\; \pi_* \left( \sum_{j=-n} \alpha_{n+j}H^{n+j} \right) \\
=&\sum_{j=-n} \alpha_{n+j} \cdot  \pi_*\left(H^{n+j} \right) \;\;\; (\mathrm{By \; the \; projection \; formula}) \\
\overset{\eqref{eq:piH12}}=& \sum_{j=0} \alpha_{n+j} \left( \sum_{i_1+\cdots+i_m=j} \beta_{i_1}^{(k_1)}(-L_1)^{i_1}\cdot \beta_{i_2}^{(k_2)}(-L_2)^{i_2} \cdots \beta_{i_m}^{(k_m)}(-L_m)^{i_m} \right)\;\;\; \\
=& \sum_{j=0} \alpha_{n+j} \left(\sum_{i_1+\cdots+i_m=j} \frac{1}{k_1!} \frac{d^{k_1}}{dx_1^{k_1}} \left( x_1^{i_1+k_1}\right) \bigg|_{x_1=-L_1} \cdots \frac{1}{k_m!} \frac{d^{k_m}}{dx_m^{k_m}} \left( x_m^{i_m+k_m}\right) \bigg|_{x_m=-L_m}  \right) \;\;\;  \\
=& \sum_{j=0} \alpha_{n+j} \left. \left( \frac{1}{k_1!\cdots k_m!} \frac{\partial^{k_1+\cdots +k_m}}{\partial x_1^{k_1}\cdots \partial x_m^{k_m}}  \sum_{i_1+\cdots+i_m=j}  x_1^{i_1+k_1}\cdots x_m^{i_m+k_m}  \right)\right|_{
x_1=-L_1, \dots,x_m=-L_m }  \\
=&\sum_{j=0} \alpha_{n+j} \left. \left( \frac{1}{k_1!\cdots k_m!} \frac{\partial^{k_1+\cdots +k_m}}{\partial x_1^{k_1}\cdots \partial x_m^{k_m}}  \left(x_1^{k_1} \cdots x_m^{k_m}\sum_{i_1+\cdots+i_m=j}  x_1^{i_1}\cdots x_m^{i_m}\right)  \right)\right|_{
x_1=-L_1, \dots,x_m=-L_m }  \\
\overset{(\ref{spi})}=&\sum_{j=0} \alpha_{n+j} \left. \left( D\cdot \sum_{i=1}^m \frac{x_i^{m+j-1}}{\prod_{l=1, \; l \neq i}^{m} (x_i-x_l)}  \right)\right|_{x_1=-L_1, \dots,x_m=-L_m } \;\;\; \\
=&\left.\left(D\cdot \sum_{i=1}^m \frac{\sum_{j=0} \alpha_{n+j} x_i^{m+j-1}}{\prod_{l=1, \; l \neq i}^{m} (x_i-x_l)}  \right)\right|_{
x_1=-L_1, \dots,x_m=-L_m } \\
=&\left.\left(D\cdot \sum_{i=1}^m \frac{\alpha(x_i) -\left(\alpha_0+\alpha_1 x_i+\cdots + \alpha_{\text{rank}(\mathscr{E})-2}x_i^{\text{rank}(\mathscr{E})-2} \right)}{ x_i^{\text{rank}(\mathscr{E})-m}\prod_{l=1, \; l \neq i}^{m} (x_i-x_l)}\right)\right|_{
x_{1}=-L_{1},...,x_m=-L_m}.
\end{align*}
 
The conclusion follows from the last expression; note that in the seventh equality we used the  following fact (see \cite[Theorem~3.2]{CI}) 
\begin{equation}\label{spi}
 \sum_{i_1+\cdots+i_m=j}  x_1^{i_1}x_2^{i_2} \cdots x_m^{i_m} =\sum_{i=1}^m \frac{x_i^{m+j-1}}{\prod_{l=1, \; l \neq i}^{m} (x_i-x_l)}.
\end{equation}
\end{proof}

Clearly formula (\ref{mf1}) is symmetric in the distinct Chern roots of $\mathscr{E}$, so that the end result may be expressed in terms of the Chern classes $c_i(\mathscr{E})$ and the $\alpha_j$ for $j\geq n$. As we view the operation $\pi_*$ as a relative form of integration, it is interesting to note that the appearance of the differential operator on the right-hand side of formula (\ref{mf1}) gives it the flavor of a `Stokes-type formula'. It follows from formula (\ref{spi}) (along with degree considerations) that terms of the form  $L_i^{\text{rank}(\mathscr{E})-m}\prod_{l=1, \; l \neq i}^{m} (L_i-L_l)$ appearing in final expression for $\pi_*\alpha$ in fact cancel once put over a common denominator. As such, in practice one may in fact avoid the use of the differential operator $D$ by first assuming \emph{all} the non-trivial Chern roots are  distinct, and then evaluating the final simplified rational expression at the possibly non-distinct total number of non-trivial Chern roots of $\mathscr{E}$.  
\begin{remark}
Similar to the discussion above, if one changes the notational convention so that the last (or first) Chern root is assumed to be zero and makes no assumptions on the remainder of the roots, then the $L_i$'s in Theorem~\ref{pff} can be treated as formal symbols and zero values can be substituted in after clearing denominators in the pushforward resulting from the formula in Theorem~\ref{pff}. 
In particular, suppose we have a vector bundle $\mathscr{E}\to X$ of rank $r$ with Chern roots $\alpha_1, \dots, \alpha_{r}$. If, as above, we choose $\alpha_{r}$ and set $\mathscr{L}\to X$ to be a line bundle whose first Chern class is $-\alpha_{r}$, then after tensoring $\mathscr{E}$ with $\mathscr{L}$, we get a vector bundle (which can still be called $\mathscr{E}$, as this yields an isomorphic projectivization $\mathbb{P}(\mathscr{E})$) with Chern roots $L_1,\dots, L_{r-1}$
and $L_{r} = 0$. We could then \textit{define} $L_1,\dots, L_{r-1}$ as non-trivial Chern roots (regardless of their values, and allowing them to be zero). With this convention we would apply the formula of Theorem~\ref{pff} treating the Chern roots  $L_1,\dots, L_{r-1}$ as formal symbols and evaluating them after clearing denominators.  
In this way one could avoid assuming the Chern roots in Theorem~\ref{pff} are non-zero. However, since only the non-zero Chern roots determine the pushforward (for non-trivial bundles), we have opted for the notational convention which assumes the $L_i$'s in Theorem~\ref{pff} are non-zero.
\label{remark:OtherNotationalConvention}\end{remark}
For projective bundles of small rank, formula (\ref{mf1}) is reasonable to apply by hand, though in any case a computer implementation is straightforward and is of negligible cost. We note that such an implementation may be of utility for software packages computing characteristic classes of varieties. In particular an implementation of (\ref{mf1}) allows the applicability setting of algorithms for computing singular Chern classes and Euler characteristics of (possibly singular) projective varieties and schemes (such as \cite{jost2013algorithm, Helmer2015}) to be broadened to the setting of projective bundles. To this end, and to simplify computations for the interested reader, we provide a test implementation in the form of a Macaulay2 \cite{M2} package which can be downloaded from \url{https://github.com/Martin-Helmer/PBF}. 

We note that when $\mathscr{E}$ is the direct sum of tensor powers of a fixed line bundle on $X$ we recover the main formula derived in \cite[Theorem~1.1]{F}, and when $\mathscr{E}$ is the direct sum of three distinct line bundles we recover a formula appearing in \cite[\S2.7]{EJY}.

\section{Relative Chern class}\label{rccc}
We recall that in \S\ref{fcc} we saw that if $\varphi:Y_X\to X$ is a relative variety, then 
\begin{equation}\label{rgby}
\chi(Y_X)=\int_X \varphi_*c(Y_X),
\end{equation}
so that explicitly computing the RHS of equation (\ref{rgby}) yields a `Sethi-Vafa-Witten formula' for the Euler characteristic of $Y_X$ in terms of a Chow class on $X$. In this section, we discuss the precise sense in which $\varphi_*c(Y_X)$ is the Chern class of the \emph{morphism} $\varphi:Y_X\to X$, and then use the functorial formalism of \S\ref{fcc} to derive a formula for $\varphi_*c(Y_X)$ when $Y_X$ is the projecto-relative hypersurface $Z_X(n,d)$ defined in \S\ref{rv}. An explicit formula for $Y_X$ an arbitrary projecto-relative hypersurface is then derived in section \S\ref{SVW}.

Let $V$ be a variety and denote by $K_0(\text{Var}_{V})$ the free $\mathbb{Z}$-module generated by isomorphism classes of proper morphisms to $V$, modulo the relation 
\[
[Y\overset{\psi}\to V]=[Z\overset{\left.\psi\right|_{Z}}\longrightarrow V]+[U\overset{\left.\psi\right|_{U}}\longrightarrow V],
\]
whenever $Z$ is a closed subvariety of $Y$ with open complement $U$. We may endow $K_0(\text{Var}_{V})$ with a ring structure by setting
\[
[Y\overset{\psi}\to V]\cdot [W\overset{\vartheta}\to V]=[Y\times_{V}W\to V],
\]
where $Y\times_{V}W\to V$ denotes the fibered product of $Y$ with $W$ over $V$ (with respect to the morphisms $\psi$ and $\vartheta$). We refer to the ring $K_0(\text{Var}_{V})$ as the \emph{relative Grothendieck ring of varieties} (over $V$). When $V=\star$ is a point, the relative Grothendieck ring is often denoted $K_0(\text{Var}_{\C})$, and referred to simply as the \emph{Grothendieck ring of varieties}. Now given a proper morphism $V\overset{f}\to S$, we define an associated homomorphism
\[
K_0(f):K_0(\text{Var}_{V})\longrightarrow K_0(\text{Var}_{S})
\]
given by
\[
[Y\to V]\mapsto [Y\to V\to S].
\]
With such prescriptions $K_0(\text{Var}_{(\cdot)})$ is then a covariant functor from varieties to groups with respect to proper morphisms, which we refer to as the \emph{relative Grothendieck functor}. Moreover, there exists a natural transformation of functors
\[
\nu:K_0(\text{Var}_{(\cdot)})\to F,
\]
given by
\[
\nu([Y\overset{\psi}\to V])=\psi_*\mathbbm{1}_Y\in F(V),
\]
where we recall that $\psi_*\mathbbm{1}_Y$ is the pushforward to $V$ of the constructible function $\mathbbm{1}_Y$ (as prescribed by equation (\ref{pffcf})). Then composing $\nu:K_0(\text{Var}_{(\cdot)})\to F$ with the MacPherson natural transformation $c_*:F\to A_*$ yields
\begin{equation}\label{e9}
c_*(\psi_*\mathbbm{1}_Y)=\psi_*(c_*(\mathbbm{1}_Y))=\psi_*c_{\text{SM}}(Y)\in A_*V.
\end{equation}
As such, after denoting the the composition $c_*\circ \nu$ simply by $c$, we have 
\[
c([Y\overset{\psi}\to V])=\psi_*c_{\text{SM}}(Y),
\]
which we refer to as the \emph{relative Chern class} of the morphism $\psi:Y\to V$ (though we should really say relative Chern class of the \emph{class} of the morphism $\psi:Y\to V$ in the relative Grothendieck ring), which we will often denote by $c(\psi)$. In such a framework, the Chern-Schwartz-MacPherson class of a variety $V$ is the relative Chern class of the identity morphism $\text{id}_V:V\to V$, i.e.,
\[
c_{\text{SM}}(V)=c(\text{id}_V)\in A_*V
\]
(for more on relative characteristic classes see \cite{brasselet2010hirzebruch,bussi2014naive}).

Now let $X$ be a smooth compact base variety and let $\varphi:Z_X(n,d)\to X$ be the projecto-relative hypersurface defined in \S\ref{rv}, namely, the projecto-relative hypersurface given by
\[
Z_X(n,d):(x_1^d+\cdots+x_n^d+fx_1x_0^{d-1}+gx_0^d=0)\subset \mathbb{P}(\mathscr{E}),
\]
where $\mathscr{E}\to X$ is the direct sum of the trivial line bundle $\mathscr{O}\to X$ with $n$-copies of an ample line bundle $\mathscr{L}\to X$. For ease of notation we denote $Z_X(n,d)$ simply by $Z_X$ from here on. By equations (\ref{e9}) we have
\[
c(\varphi)=c_*\left(\varphi_*\mathbbm{1}_{Z_X}\right).
\]
We now focus on the computation of $\varphi_*\mathbbm{1}_{Z_X}$. Recall that the discriminant of $Z_X$ is given by 
\[
\Delta_d:((d-1)^{d-1}f^d+d(-d)^{d-1}g^{d-1}=0)\subset X,
\]
and for $Z_X$ to be smooth we make the assumption that the hypersurfaces
\[
F:(f=0)\subset X \quad \text{and} \quad G:(g=0)\subset X
\]
are smooth and intersect transversally. This further implies that the codimension 2 complete intersection
\[
C:(f=g=0)\subset X
\]
is smooth as well. We can then stratify $X$ according to the fibers of $\varphi:Z_X\to X$. The fibers of $\varphi$ over $X\setminus \Delta_d$ are smooth degree $d$ hypersurfaces whose homeomorphism class we denote by $F_0$, the fibers over $\Delta_d\setminus C$ are singular fibers whose homeomorphism class we denote by $F_1$, and the fibers over $C$ are singular fibers whose homeomorphism class we denote by $F_2$. By the definition of pushforward of constructible functions given by equation (\ref{pffcf}) we  have
\begin{eqnarray*}
\varphi_*\mathbbm{1}_{Z_X}&=&\chi(F_0)(\mathbbm{1}_X-\mathbbm{1}_{\Delta_d})+\chi(F_1)(\mathbbm{1}_{\Delta_d}-\mathbbm{1}_C)+\chi(F_2)\mathbbm{1}_C \\
                                        &=&\chi(F_0)\mathbbm{1}_X+(\chi(F_1)-\chi(F_0))\mathbbm{1}_{\Delta_d}+(\chi(F_2)-\chi(F_1))\mathbbm{1}_C,  \\
\end{eqnarray*}
so that
\begin{eqnarray*}
c(\varphi)&=&c_*\circ \varphi_*(\mathbbm{1}_{Z_X}) \\
             &=&c_*(\chi(F_0)\mathbbm{1}_X+(\chi(F_1)-\chi(F_0))\mathbbm{1}_{\Delta_d}+(\chi(F_2)-\chi(F_1))\mathbbm{1}_C) \\
             &=&\chi(F_0)c_{\text{SM}}(X)+(\chi(F_1)-\chi(F_0))c_{\text{SM}}(\Delta_d)+(\chi(F_2)-\chi(F_1))c_{\text{SM}}(C). \\
\end{eqnarray*}
Now since $X$ and $C$ are smooth, their CSM-classes coincide with their usual Chern classes $c(X)$ and $c(C)$, which are given by
\begin{equation}\label{e10}
c(X)=c(TX)\cap [X] \quad \text{and} \quad c(C)=c(TX)\cap \left(\frac{[F]\cdot [G]}{(1+[F])(1+[G])}\right).
\end{equation}
The details of the computation of $c_{\text{SM}}(\Delta_d)$ would take us too far afield (i.e., to the theory of Segre classes of singular schemes of hypersurfaces), but $c_{\text{SM}}(\Delta_d)$ was computed in \cite[\S4.2]{AE1} for the case $d=3$, and the computation for general $d$ is no different. As such, we have\small
\begin{equation}\label{e11}
c_{\text{SM}}(\Delta_d)=c(TX)\cap \left(\frac{[\Delta_d]}{1+[\Delta_d]}+\frac{(d-2)(d-1)[F]\cdot [G]}{(1+[\Delta_d])(1+[\Delta_d]+(1-d)[F])(1+[\Delta_d]+(2-d)[G])}\right),
\end{equation}\normalsize
thus we have an explicit formula for $c(\varphi)$ once we know the Euler characteristics of $F_0$, $F_1$ and $F_2$. The Euler characteristic of a smooth hypersurface of degree $d$ in $\mathbb{P}^n$ is a polynomial in $d$, namely
\begin{equation}\label{epf}
\mathfrak{e}_n(d)=-\sum_{k=0}^{n-1}\binom{n+1}{k}(-d)^{n-k},
\end{equation}
so that $\chi(F_0)=\mathfrak{e}_n(d)$. The singular fibers $F_1$ and $F_2$ both have an isolated singular point whose Milnor numbers are given by $(-1)^{n}(d-1)^{n-1}$ and $(-1)^n(d-1)^n$ respectively, thus their Euler characteristics are given by
\[
\chi(F_1)=\mathfrak{e}_n(d)+(-1)^{n}(d-1)^{n-1}, \quad \chi(F_2)=\mathfrak{e}_n(d)+(-1)^n(d-1)^n.
\]
We then have
\[
c(\varphi)=\mathfrak{e}_n(d)\cdot c(X)+(-1)^{n}(d-1)^{n-1}\cdot c_{\text{SM}}(\Delta_d)+(-1)^{n}((d-1)^n-(d-1)^{n-1})\cdot c(C),
\]
which after a lot of messy algebra yields
\begin{equation}\label{STMF1}
c(\varphi)=\frac{\mathfrak{e}_n(d)+(\mathfrak{e}_{n-1}(d)+1)dL }{1+dL}c(\text{id}_X),
\end{equation}
where $L$ denotes the first Chern class of the ample line bundle $\mathscr{L}\to X$ used in the definition of $Z_X(n,d)$. 

We note that even though $Z_X$ is smooth, we needed to use singularity theory to compute $c(\varphi)$ with this method, as it relied on the geometry of both the singular fibers of the fibration $\varphi:Z_X\to X$ and the geometry of its singular discriminant $\Delta_d$. In the next section, we use the pushforward formula of Theorem~\ref{pff} to bypass singularity theory in the computation of $c(\varphi)$. In particular, in using Theorem~\ref{pff}, we will show that all that is needed to compute the pushforward of $c(Y_X)$ for an arbitrary projecto-relative hypersurface $Y_X\to X$ are simple algebraic manipulations of rational expressions determined by its rational equivalence class in $A_*\mathbb{P}(\mathscr{E})$.

\section{Sethi-Vafa-Witten formulas}\label{SVW} 
We now use the pushforward formula of Theorem~\ref{pff} to compute $c(\varphi)$ for an arbitrary smooth projecto-relative hypersurface. So let $\varphi:Y_X\to X$ be a smooth projecto-relative hypersurface over a smooth compact base variety $X$, so that $Y_X$ is a hypersurface in a $\mathbb{P}^n$-bundle $\pi:\mathbb{P}(\mathscr{E})\to X$ for some rank $r$ vector bundle $\mathscr{E}\to X$. By the structure theorem for the Chow group of a projective bundle (cf.~\cite[Theorem~3.3]{IT}), the class of $Y_X$ in $A_*\mathbb{P}(\mathscr{E})$ is of the form $[Y_X]=dH+\pi^*\beta$, where $d$ is the degree of the equation defining $Y_X$, $\beta$ is a divisor in $X$, and $H$ is the class associated with the line bundle $\mathscr{O}(1)\to \mathbb{P}(\mathscr{E})$. We then associate with $Y_X$ a rational expression $Q_d(\mathscr{E},\beta)$ in the class $\beta\in A_*X$ and the non-trivial Chern roots $(L_1,\ldots, L_n)$ of $\mathscr{E}$, which for the moment we assume to be distinct. The rational expression $Q_d(\mathscr{E},\beta)$ -- which will be seen to arise from the application of Theorem \ref{pff} to the current setting in proof of Theorem \ref{rcc} below -- is then given by
\begin{equation}
Q_d(\mathscr{E},\beta)=\sum_{i=1}^{n}  \left(\frac{(1-L_i)\prod_{j=1}^n (1+L_j-L_i)}{{L_i \prod_{j=1, \; j\neq i}^n (L_i-L_j)}} \cdot \frac{\beta-d L_i}{1+\beta-d L_i}-\frac{\prod_{j=1}^n (1+L_j)}{L_i \prod_{j=1, \; j\neq i}^n (L_i-L_j)} \cdot \frac{\beta}{1+\beta}\right)\label{eq:Q_dEbeta}
\end{equation}
where we recall $n=$ is the number of non-trivial Chern roots of $\mathscr{E}$. If the non-trivial Chern roots are not all distinct, we may treat the $L_i$ as formal variables which may be evaluated at the non-distinct Chern roots after putting the expression for $Q_d(\mathscr{E},\beta)$ over a common denominator and canceling common factors from the numerator and denominator of the form $L_i \prod_{j=1, \; j\neq i}^r (L_i-L_j)$. In doing so, the expression for $Q_d(\mathscr{E},\beta)$ then takes the form
\begin{equation}\label{re}
Q_d(\mathscr{E},\beta)=\frac{P(\mathscr{E},\beta)}{(1+\beta)\cdot(1+\beta-dL_1)^{k_1+1}\cdots(1+\beta-dL_m)^{k_m+1}},
\end{equation}
where $L_1,...,L_m$ are the \emph{distinct} non-trivial Chern roots of $\mathscr{E}$, $k_i+1$ is the multiplicity of $L_i$, and $P(\mathscr{E},\beta)$ is a polynomial in the non-trivial Chern roots of $\mathscr{E}$ and $\beta \in A_*X$. {For an example of these see the discussion proceeding \eqref{eq:Qnd_BinomialForm}.}

We note that if $\mathscr{E}_d^{\vee}$ denotes a vector bundle whose distinct non-trivial Chern roots are $-dL_1,...,-dL_m$, where $-dL_i$ has the same multiplicity as $L_i$ (i.e.~$k_i+1$), then the denominator of $Q_d(\mathscr{E},\beta)$ as expressed in (\ref{re}) may be associated with the Chern class of the vector bundle $\mathscr{E}_d^{\vee}\otimes \mathscr{O}(\beta)$, so that $Q_d(\mathscr{E},\beta)$ may be written as
\[
Q_d(\mathscr{E},\beta)=c(\mathscr{E}_d^{\vee}\otimes \mathscr{O}(\beta))^{-1}\cap P(\mathscr{E},\beta)\in A_*X.
\] 
The Chern class of $\varphi$ is then given by the following
\begin{theorem}\label{rcc}
Let $Y_X\hookrightarrow \mathbb{P}(\mathscr{E})$ be a smooth projecto-relative hypersurface of degree $d$, and let $\varphi:Y_X\to X$ denote the associated projection. Then
\begin{equation}
c(\varphi)=Q_d(\mathscr{E},\beta)\cdot c(\emph{id}_X)\in A_*X.
\end{equation}
\end{theorem}

 Before proving Theorem~\ref{rcc}, we first prove the following
\begin{lemma}\label{L1}
Let $R$ be a commutative integral domain with unity and let $m\in \mathbb{N}$ be a positive integer. Then for any choice of $\alpha_0, \alpha_1, \ldots, \alpha_{m-1}\in R$ we have
\begin{equation}\label{SFE}
\sum_{i=1}^m  \frac{\alpha_0+ \alpha_1 x_i + \cdots + \alpha_{m-1}x_i^{m-1}}{x_i\prod_{j=1, \; j\neq i}^m (x_i-x_j)}=\sum_{i=1}^m \frac{\alpha_0}{x_i\prod_{j=1, \; j\neq i}^m (x_i-x_j)}\in R(x_1,...,x_m).
\end{equation}
\end{lemma}
\begin{proof}
The left-hand side of equation (\ref{SFE}) may be expanded as
\[
\sum_{i=1}^m \alpha_0\frac{x_i^{-1}}{\prod_{j=1, \; j\neq i}^m (x_i-x_j)}+ \alpha_1\sum_{i=1}^m\frac{1} {\prod_{j=1, \; j\neq i}^m (x_i-x_j)}+ \cdots + \alpha_{m-1}\sum_{i=1}^m \frac{x_i^{m-2}}{\prod_{j=1, \; j\neq i}^m (x_i-x_j)},
\]
and in (i) of the proof of Theorem~3.2 of \cite{CI} it is shown that for  $0\leq p \leq m-2$ we have
\begin{equation}
    \sum_{i=1}^m \frac{x_i^{p}}{\prod_{j=1, \; j\neq i}^m (x_i-x_j)}=0,\label{eq:identityLem_6_2}
\end{equation}
from which the lemma follows. We note that there is in fact a typo in formula (i) of Theorem 3.2 of \cite{CI}, however we have corrected the typo in formula \eqref{eq:identityLem_6_2} above. In particular, in \cite{CI} they have incorrectly switched the bounds between the range values denoted $p,m$ in our notation.  
\end{proof}

\begin{proof}[Proof of Theorem~\ref{rcc}]
It suffices to prove the case where the non-trivial Chern roots of $\mathscr{E}$ are all distinct, since one may set them equal to each other after the final result if some of them appear with multiplicity. In such a case, we have $m=n$. Now let $\mathscr{O}(Y_X)\to \mathbb{P}(\mathscr{E})$ denote the line bundle corresponding to the divisor class of $Y_X$ (whose restriction to $Y_X$ yields the normal bundle to $Y_X$ in $\mathbb{P}(\mathscr{E})$). In order to write an explicit expression for $c(\varphi)$, we first write an expression for $c(\iota)$, where $\iota:Y\to \mathbb{P}(\mathscr{E})$ denotes the natural inclusion. For this, we consider the following exact sequences of vector bundles (as Chern classes are multiplicative) 
\[
0 \to TY_X \to \iota^*T\mathbb{P}(\mathscr{E}) \to \iota^*\mathscr{O}(Y_X) \to 0
\]
\[
0\to T_{\mathbb{P}(\mathscr{E})/X} \to T\mathbb{P}(\mathscr{E}) \to \pi^* TX \to 0
\]
\[
0 \to \mathscr{O} \to \pi^*(\mathscr{E}) \otimes \mathscr{O}(1) \to T_{\mathbb{P}(\mathscr{E})/X} \to 0.
\]
The first sequence (from the top) is standard, the second one may take as the definition of the relative tangent bundle $T_{\mathbb{P}(\mathscr{E})/X}$, and the third is given in B.5.8 of \cite{IT}. We then have
\begin{eqnarray*}
c(\iota)&=&\iota_*\left(c(TY_X)\cap [Y_X]\right) \\
        &=&\iota_*\left(\frac{c(\iota^*T\mathbb{P}(\mathscr{E}))}{c(\iota^*\mathscr{O}(Y_X))} \cap [Y_X]\right) \\
                       &=&\frac{c(\pi^*( \mathscr{E}) \otimes \mathscr{O}(1))\cdot c(\pi^*TX)}{c(\mathscr{O}(Y_X))} \cap \iota_*[Y_X] \\
                       &=&\frac{c(\pi^*( \mathscr{E}) \otimes \mathscr{O}(1))\cdot c(\pi^*TX)}{c(\mathscr{O}(Y_X))} \cap \left(c_1(\mathscr{O}(Y_X))\cap \pi^*[X]\right) \\
											 &=&\frac{c(\pi^*( \mathscr{E}) \otimes \mathscr{O}(1))\cdot c_1(\mathscr{O}(Y_X))}{c(\mathscr{O}(Y_X))} \cap \left(c(\pi^*TX)\cap \pi^*[X]\right) \\
											 &=&\frac{c(\pi^*( \mathscr{E}) \otimes \mathscr{O}(1))\cdot c_1(\mathscr{O}(Y_X))}{c(\mathscr{O}(Y_X))} \cap \pi^*c(\text{id}_X), \\
\end{eqnarray*}
where the third equality follows from the projection formula, the fourth from the fact that $\pi^*[X]=[\mathbb{P}(\mathscr{E})]$ (since $\iota_*[Y_X]=c_1(\mathscr{O}(Y_X))\cap [\mathbb{P}(\mathscr{E})]$), the fifth from the commutativity of Chern classes, and the sixth by the behavior of Chern classes with respect to pullback.

Now since $(0,L_1,...,L_m)$ are the Chern roots of $\mathscr{E}$ (which we recall we have assumed to be distinct), we have that the Chern roots of $\pi^*(\mathscr{E}) \otimes \mathscr{O}(1)$ are then $(H,L_1+H,\dots,L_m+H)$, (note that technically, the Chern roots are of the form $\pi^*L_i+H$ but we have omitted the pullbacks).
Thus 
\[
c(\iota)=\frac{(1+H)(1+H+L_1)\cdots (1+H+L_m)\cdot [Y_X]}{1+[Y_X]}\cdot \pi^*c(\text{id}_X),
\]
so that
\begin{eqnarray*}
c(\varphi)&=&\varphi_*c(\text{id}_{Y_X}) \\
                      &=&\pi_*\circ \iota_*c(\text{id}_{Y_X}) \\
											&=&\pi_*c(\iota) \\
											&=&\pi_*\left(\frac{(1+H)(1+H+L_1)\cdots (1+H+L_m)\cdot [Y_X]}{1+[Y_X]}\cdot \pi^*c(\text{id}_X)\right) \\
										  &=&\pi_*\left(\frac{(1+H)(1+H+L_1)\cdots (1+H+L_m)\cdot [Y_X]}{1+[Y_X]}\right)\cdot c(\text{id}_X), \\
\end{eqnarray*}
where the second equality follows from functoriality and the fifth follows via the projection formula. The theorem then follows once we show
\[
\pi_*\left(\frac{(1+H)(1+H+L_1)\cdots (1+H+L_m)\cdot [Y_X]}{1+[Y_X]}\right)=Q_d(\mathscr{E},\beta),
\] 
where we recall that the class of $Y_X$ is of the form $[Y_X]=dH+\pi^*\beta$ for some divisor $\beta$ in $X$. So now let  $\alpha(H)$ be given by
\[
\alpha(H)=\frac{\left(1+H\right)\cdot \prod_{i=1}^m\left(1+H+L_i \right)\cdot [Y_X]}{1+[Y_X]}.
\]
Applying Theorem~\ref{pff} along with Lemma~\ref{L1} to the class $\alpha(H)$ then yields 
\begin{eqnarray*}
\pi_*\alpha(H)&=&\sum_{i=1}^m \left(\frac{ \alpha(x_i)-\alpha_0}{x_i\prod_{j=1, \; j\neq i}^m (x_i-x_j)}\right)\bigg|_{x_1=-L_1,\dots,x_m=-L_m} \label{eq:QInProof} \\
           &=&{\sum_{i=1}^{m}  \left(\frac{(1-L_i)\prod_{j=1}^m (1+L_j-L_i)}{{L_i \prod_{j=1, \; j\neq i}^m (L_i-L_j)}} \cdot \frac{\beta-d L_i}{1+\beta-d L_i}-\frac{\prod_{j=1}^m (1+L_j)}{L_i \prod_{j=1, \; j\neq i}^m (L_i-L_j)} \cdot \frac{\beta}{1+\beta}\right)}\\
\end{eqnarray*}
 which coincides with $Q_d(\mathscr{E},\beta)$ by definition, thus concluding the proof.
\end{proof}

We note that for concrete realizations of $Y_X$, the associated pushforward formula and the form of $Q_d(\mathscr{E},\beta)$ will often be quite simple. In particular consider the case $Y_X=Z_X(n,d)$ as defined in \S\ref{rv} and let $Q_{n,d}(L)=Q_d(\mathscr{E},\beta)$ be the resulting expression in $L$, $n$ and $d$. First compute $Q_d(\mathscr{E},\beta)$ for $Y_X=Z_X(2,d)$ using Theorem \ref{pff}. We treat the repeated Chern root $L$ as two distinct Chern roots $L_1,L_2$. We now apply equation \eqref{eq:Q_dEbeta} with the two distinct Chern roots. The summation in the first term is:$$
\frac{\left({L}_{1}-1\right)\left({L}_{1}-{L}_{2}-1\right)\left(d{L}_{1}-\beta\right)(\beta+1)+{\beta}\left({L}_{2}+1\right)\left({L}_{1}+1\right)(dL_1-(\beta+1))}{L_1(L_1-L_2)(\beta+1)(dL_1-(\beta+1))},
$$ and the second term in the summation is:$$
\frac{\left({L}_{2}-1\right)\left({L}_{1}-{L}_{2}+1\right)\left(d{L}_{2}-\beta\right)(\beta+1)-{\beta}\left({L}_{2}+1\right)\left({L}_{1}+1\right)(dL_2-(\beta+1))}{L_2(L_1-L_2)(\beta+1)(dL_2-(\beta+1))}.
$$
Adding these two terms we see that each term in the resulting numerator has a factor of $(L_1L_2(L_1-L_2))$, this term also appears in the denominator. Canceling these from the numerator and denominator we obtain $Q_d(\mathscr{E},\beta)=\frac{P(L_1,L_2,d,\beta)}{(1+\beta)(dL_1-(1+\beta))(dL_2-(1+\beta))}$ where the numerator $P(L_1,L_2,d,\beta)$ is:
\tiny $$
{3L_1L_2d^2\beta+2L_1L_2d^2-3L_1d\beta^2-3L_2d\beta^2-L_1d^2-L_2d^2-4L_1d\beta-4L_2d\beta+3\beta^3-L_1d-L_2d-d^2+3d\beta+6\beta^2+3d+3\beta}.
$$\normalsize
Substituting in $L_1=L_2=L$ and $\beta=dL$ we obtain:$$Q_{2,d}(L)=
\frac{{L}d^{2}+{L}d-d^{2}+3\,d}{d{L}+1}=\frac{dL(d+1)+d(3-d)}{dL+1}.
$$Now consider the case $Y_X=Z_X(3,d)$; applying \eqref{eq:Q_dEbeta} with the three distinct Chern roots $L_1,L_1,L_3$, simplifying with a computer algebra system, and substituting in $L_1=L_2=L_3=L$ and $\beta=dL$ we obtain: $$Q_{3,d}(L)=
\frac{-L\,d^{3}+3\,L\,d^{2}+d^{3}+L\,d-4\,d^{2}+6\,d}{dL+1}=\frac{dL(-d^2+3d+1)+d(d^2-4d+6)}{dL+1}.
$$ By induction we see that \begin{equation}
 Q_{n,d}(L)=\frac{dL\left(1+ \displaystyle\sum_{i=1}^{n-1}(-1)^{i+1}{n \choose i+1}d^i \right)+d\left( \displaystyle\sum_{i=0}^{n-1}(-1)^i {n+1\choose i+2}d^i\right)}{1+dL}.\label{eq:Qnd_BinomialForm}   
\end{equation} 
In particular, for the case $Y_X=Z_X(n,d)$ as defined in \S\ref{rv}, we have
\[
Q_d(\mathscr{E},\beta)=\frac{\mathfrak{e}_n(d)+(\mathfrak{e}_{n-1}(d)+1)dL }{1+dL}.
\]
This reconfirms the derivation of \eqref{STMF1} using singularity theoretic methods.
We remark that using the methods of \S\ref{rccc} to compute $c(\varphi)$ will quickly become impractical as one considers more complicated examples since the singularities of its fibers and complexity of its discriminant will worsen significantly. In fact, in many such examples just computing the discriminant itself can be unfeasible, let alone its CSM class. On the other hand, computing $c(\varphi)$ via the pushforward formula remains the same, as it reduces the computation of $c(\varphi)$ to simple algebraic manipulations which are irrespective of the singularities of the fibers and discriminant of $\varphi:Y_X\to X$.  

Now let $Q_t$ be the polynomial in $t$ which comes from evaluating $Q_d(\mathscr{E},\beta)$ at $H=tH$, $\beta=t\beta$ and $L_i=tL_i$ for $i=1,...,n$, and denote by $c_t(X)$ the Chern polynomial of $X$, so that
\[
c_t(X)=1+c_1(X)t+c_2(X)t^2+\cdots +c_{\text{dim}(X)}(X)t^{\text{dim}(X)},
\]
and denote the map which takes a polynomial in $t$ to its $k$th coefficient by $[t^k]$. An immediate corollary of Theorem~\ref{rcc} is then given by the following.
\begin{corollary}\label{ecc}
 The Euler characteristic of $Y_X$ is given by the Sethi-Vafa-Witten formula
\[
\chi(Y_X)=\int_X [t^{\emph{dim}(X)}]\left(Q_t\cdot c_t(X)\right).
\]
\end{corollary}
\begin{proof}
Rewriting the relative Gauss-Bonnet formula \eqref{svwf} using the notations of \S\ref{rccc} yields
\[
\chi(Y_X)=\int_X c(\varphi),
\]
and since
\[
c(\varphi)=\sum_{k=0}^{\text{dim}(X)}[t^k]\left(Q_t\cdot c_t(X)\right),
\]
the result follows from the definition of the map $\int_X:A_*X\to \mathbb{Z}$.
\end{proof}
We now apply Corollary \ref{ecc} to $Y_X=Z_X(n,d)$ as defined in \S\ref{rv}. Set $m=\dim(X)$, we have \small\begin{equation}
\chi(Z_X(n,d))=\mathfrak{e}_n(d)c_{m}(X)+(-dL)^{m}(\mathfrak{e}_n(d)-\mathfrak{e}_{n-1}(d)-1)+\sum_{i=1}^{m-1}(-dL)^{i}(\mathfrak{e}_n(d)-\mathfrak{e}_{n-1}(d)-1)c_{m-i}(X).
\end{equation}\normalsize For a second concrete illustration, we apply Corollary~\ref{ecc} to the elliptic fibration introduced in \S\ref{intro}.
\begin{example}
The Weierstrass equation
\[
y^2z=x^3+fxz^2+gz^3
\]
defines a projecto-relative hypersurface $\varphi:Y_X\to X$ by taking $x$ as a section of $\mathscr{O}(1)\otimes \pi^*\mathscr{L}^2$, $y$ a section of $\mathscr{O}(1)\otimes \pi^*\mathscr{L}^3$, $z$ a section of $\mathscr{O}(1)$, $f$ a section of $\pi^*\mathscr{L}^4$ and $g$ a section of $\pi^*\mathscr{L}^6$ (with $\mathscr{L}\to X$ an ample line bundle). In such a case, we have 
\[
\mathscr{E}=\mathscr{O}\oplus \mathscr{L}^2\oplus \mathscr{L}^3,
\]
and $[Y_X]=3H+6L$ (where $L$ is the Chern root of $\mathscr{L}\to X$), so that $\beta$ in this case is $6L$. The rational expression $Q_3(\mathscr{E},6L)$ is then given by
\[
Q_3(\mathscr{E},6L)=\frac{L-1}{L}+\frac{(1+2L)(1+3L)(6L)}{6L^2(1+6L)}=\frac{12L}{1+6L}.
\] 
Applying Corollary~\ref{ecc} and taking $\mathscr{L}=-K_X$ (so that $X$ must be Fano) then recovers the Sethi-Vafa-Witten formula
\begin{equation}\label{svwf1}
\chi(Y_X)=\int_X 12c_1(X)\sum_{i=0}^{\text{dim}(X)-1} c_i(X)\left(-6c_1(X)\right)^{\text{dim}(X)-1-i},
\end{equation}
which was first proved in \cite{AE1} by the singularity-theoretic methods used in the previous section. For general $\mathscr{L}$ we have the Sethi-Vafa-Witten formula
\begin{equation}\label{svwf2}
\chi(Y_X)=\int_X 12L\sum_{i=0}^{\text{dim}(X)-1} c_i(X)\left(-6L\right)^{\text{dim}(X)-1-i}.
\end{equation}
Moreover, since 
\[
c(\varphi)=\sum_{k=0}^{\text{dim}(X)}[t^k]\left(\frac{12Lt}{1+6Lt}\cdot c_t(X)\right),
\]
we have
\begin{equation}\label{fe}
c(\varphi)=\sum_{j=1}^{\text{dim}(X)}\left(12L\sum_{i=0}^{j-1} c_i(X)\left(-6L\right)^{j-1-i}\right),
\end{equation}
so that the entire relative Chern class may be obtained by truncating the Sethi-Vafa-Witten formula (\ref{svwf2}). In particular, in the case $L=c_1(X)$, \eqref{fe} yields formula \eqref{e4} from \S\ref{intro}.
\end{example}

\noindent
{\bf Acknowledgements.}
Martin Helmer was supported by an NSERC (Natural Sciences and Engineering Research Council of Canada) postdoctoral fellowship during the preparation of this work. We are grateful to the Institute of Mathematical Research at the University of Hong Kong for hosting Martin Helmer's visit there, where part of this work was completed.

\bibliographystyle{plain}
\bibliography{ERV2}

\end{document}